\documentclass{amsproc}

\usepackage{latexsym}
\usepackage{amssymb, amsmath,mathrsfs}
\usepackage[left=2.5cm,right=2.5cm,top=2.5cm, bottom=2.5cm ]{geometry}
\usepackage[OT2,T1]{fontenc}

\usepackage{setspace}
\usepackage[pagebackref,hypertexnames=false, colorlinks, citecolor=red, linkcolor=red]{hyperref}

\newcommand{\D}{\mathbb{D}}

\DeclareMathOperator{\Rank}{Rank}
\newcommand{\Smax}{\mathcal{S}^{max}}
\newcommand{\Smin}{\mathcal{S}^{min}}
\newcommand{\Kmax}{K^{max}}
\newcommand{\Kmin}{K^{min}}

\newtheorem{theorem}{Theorem}[section]

\theoremstyle{definition}

\newtheorem{example}[theorem]{Example}

\theoremstyle{remark}
\newtheorem{remark}[theorem]{Remark}

\numberwithin{equation}{section}

\newtheorem{conjecture}[theorem]{Conjecture}
\newtheorem*{question}{Open question}

\newtheorem*{theorem*}{Theorem}
\newtheorem*{conjecture*}{Conjecture}

\begin{document}

\title[Vector-valued submodules on the bidisk]{Properties of vector-valued submodules on the bidisk}


\author[Bickel]{Kelly Bickel$^\dagger$}
\address{Department of Mathematics, Bucknell University, 360 Olin Science Building, Lewisburg, PA 17837, USA.}
\email{kelly.bickel@bucknell.edu}
\thanks{$\dagger$ Research supported in part by National Science Foundation
DMS grant \#1448846.}

\author[Liaw]{Constanze Liaw}
\address{CASPER and Department of Mathematics, Baylor University, One Bear Place \#97328, Waco, TX 76798, USA.}
\email{Constanze$\underline{\,\,\,}$Liaw@baylor.edu}


\date{\today}

\begin{abstract}
In previous work \cite{BLPreprint}, the authors studied the compressed shift operators $S_{z_1}$ and $S_{z_2}$ on two-variable model spaces $H^2(\D^2)\ominus \theta H^2(\D^2)$, where $\theta$ is a two-variable scalar  inner function. Among other results, the authors used Agler decompositions to characterize the ranks of the operators $[S_{z_j}, S^*_{z_j}]$ in terms of the degree of rational $\theta.$ In this paper, we examine similar questions for  $H^2(\D^2)\ominus \Theta H^2(\D^2)$ when $\Theta$ is a matrix-valued inner function. We extend several results from  \cite{BLPreprint} connecting $\Rank [S_{z_j}, S^*_{z_j}]$ and the degree of $\Theta$ to the matrix setting. When results do not clearly generalize, we conjecture what is true and provide supporting examples.
\end{abstract}

\maketitle

\section{Introduction}
Both Beurling's theorem on shift invariant subspaces for the Hardy space on the disk $H^2(\D)$ \cite{b48} and the model theory of Sz.-Nagy--Foia\c s (see e.g.~\cite{SzNF2010}) were of indisputable importance to central developments in function and operator theory. In this paper, we are interested in generalizations of this classical Hardy space theory to the Hardy space on the bidisk $H^2(\D^2)$. For examples, see e.g.~\cite{dg03, dg08, Rud69}. In analogy with objects important in the one-variable setting,  we consider Hilbert submodules -- namely subspaces of $H^2(\mathbb{D}^2)$ that are invariant under the Toeplitz (or shift) operators $T_{z_1}$ and $T_{z_2}$. Because of their close connections to one-variable results and the structure of inner functions, we restrict attention to submodules of Beurling-type, which are submodules of the form $\theta H^2(\D^2)$ for inner $\theta$.

Given a submodule of Beurling-type $\theta H^2(\D^2)$, one can define the associated two-variable model space $\mathcal{K}_{\theta} \equiv H^2(\mathbb{D}^2) \ominus  \theta H^2(\D^2)$. As in the one-variable setting, the compressed shift operators on these $\mathcal{K}_{\theta}$ spaces possess many interesting properties. Specifically, define
\[
S_{z_1} \equiv P_{\theta} T_{z_1}|_{\mathcal{K}_{\theta}} \ \text{ and } \  S_{z_2} \equiv P_{\theta} T_{z_2}|_{\mathcal{K}_{\theta}}, 
\]
where $P_{\theta}$ denotes the projection onto $\mathcal{K}_{\theta}.$ Interestingly, the cross commutators
$[S_{z_1}, S^*_{z_2}]$ and $[S_{z_1}, S_{z_2}]$ are related to both properties of $\theta$ and the structure of $\mathcal{K}_\theta$. See e.g.~\cite{dr00,ii06,  int04, Y01, Y02}. However, the properties of individual operators $S_{z_1}$ and $S_{z_2}$ are not as well-understood. 

One interesting result by Guo--Wang concerns rational inner functions. To state it, we first recall that the degree of a rational $\theta $ is $(m_1,m_2)$ if  $\theta=p/q$, where the polynomials $p$ and $q$ share no common factors and each $m_j$ is the maximum degree of $p$ and $q$ in $z_j$. Then, the result by Guo--Wang \cite{gw09} states that both $[S_{z_1}, S^*_{z_1}]$ and $[S_{z_2}, S^*_{z_2}]$ are compact if and only if $\theta$ is a rational inner function of degree at most $(1,1)$.   Complementing Guo--Wang's result, the authors \cite[Theorem 1.1]{BLPreprint} proved:
\begin{theorem}\label{t-rank1} Let $\theta$ be an inner function in $H^2(\mathbb{D}^2)$.
The commutator $[S_{z_1}^*,S_{z_1}]$ has rank $n$  if and only if $\theta$  is rational inner of degree $(1,n)$ or $(0,n)$.
\end{theorem}

In this paper, we seek a generalization of Theorem \ref{t-rank1} to $d\times d$ matrix-valued inner functions $\Theta$. One-variable matrix-valued inner functions appeared in the model theory of Sz.-Nagy--Foia\c s and ever since, matrix inner functions have been frequently studied alongside scalar inner functions in both the one and two variable theory. For examples, see  \cite{bsv05, bk13, RonRkN, k11, LT, Nik-Vas_model_MSRI_1998, Nik-Vas_FunctModels_1989, w10}.  Nevertheless, many proof techniques and results become much more complicated in this matrix setting.

Before discussing our results, let us introduce several standard definitions. A $d \times d$ matrix-valued function $\Theta$ is called \emph{inner} if the entries of $\Theta$ are  holomorphic functions and
\[ \Theta(\tau) \Theta(\tau)^* = \Theta(\tau)^* \Theta(\tau) = I_{d\times d} \qquad \text{ for a.e.~}\tau \in \mathbb{T}^2. \]
The vector-valued Hardy space is given by $H^2_d(\mathbb{D}^2) \equiv H^2(\mathbb{D}^2) \otimes \mathbb{C}^d$, and 
\[ \mathcal{K}_{\Theta} \equiv H^2_d(\mathbb{D}^2) \ominus \Theta H^2_d(\mathbb{D}^2)  \]
is the vector-valued model space associated to $\Theta.$ 

In this paper, we will use decompositions of vector-valued $\mathcal{K}_{\Theta}$ spaces induced via Agler kernels to study the compressed shift operators $S_{z_1}$ and $S_{z_2}$. First, recall that positive matrix-valued kernel functions $K_1,K_2: \mathbb{D}^2 \times \mathbb{D}^2 \rightarrow M_{d}(\mathbb{C})$ are called \emph{Agler kernels} of $\Theta$ if they decompose $\Theta$ as follows
 \[ I- \Theta(z) \Theta(w)^* = (1-z_1 \bar{w}_1) K_2(z,w) + (1-z_2 \bar{w}_2)K_1(z,w), \qquad \forall \ z,w \in \mathbb{D}^2. \]
J. Agler proved the existence of Agler kernels in \cite{ag90}. Subsequent work in \cite{bsv05} gave canonical constructions of Agler kernels, which were further explored in \cite{bk13}. Here is the basic setup. Define $\Smax_1$ to be the maximal $M_{z_1}$-invariant subspace of $\mathcal{K}_{\Theta}.$ Then $\Smax_1$ is the set of functions $f$ with $z_1^k f \in \mathcal{K}_{\Theta}$ for all $k\in \mathbb{N}.$ Define $\Smin_2 = \mathcal{K}_{\Theta} \ominus \Smax_1.$ It is not hard to show that there are matrix-valued kernel functions $(\Kmax_1, \Kmin_2)$ such that
 \[ \Smin_2 = \mathcal{H} \left( \frac{\Kmin_2(z,w)}{1-z_2 \bar{w}_2} \right) \text{ and }  \Smax_1 = \mathcal{H} \left( \frac{\Kmax_1(z,w)}{1-z_1 \bar{w}_1} \right), \]
where $\mathcal{H}(K)$ indicates the Hilbert space with reproducing kernel $K$. One can also show that $(\Kmax_1, \Kmin_2)$ are Agler kernels of $\Theta.$  One can similarly define shift-invariant subspaces $\Smin_1$ and $\Smax_2$ of $\mathcal{K}_{\Theta},$ which yield Agler kernels $(\Kmin_1, \Kmax_2).$ See  \cite{bk13} for details.

Our main results concern matrix-valued inner functions whose entries are also rational functions.  We say that a rational inner matrix-valued function is of degree $(m_1,m_2),$ if $m_j$ is the maximum degree of its scalar-valued entries in $z_j$, for $j=1,2.$ We also write this as $\deg_j \Theta = m_j$ for $j=1,2.$ It is worth pointing out that if $\Theta$ is a matrix-valued rational inner function, then its determinant $\det \Theta$ is a scalar rational inner function. 

\subsection{Summary of Results.} In this paper, we partially extend the results of \cite{BLPreprint} to the matrix setting. As with \cite{BLPreprint}, we first examine the situation where $\Theta$ is a product of one-variable inner functions. In the scalar setting, this study illuminated the connections between Agler kernels and compressed shift operators and provided a roadmap for obtaining  more general results. In this matrix-setting, these product inner functions are not as helpful. Indeed, rather than illuminating general results, this preliminary study illustrates that non-commutativity makes even seemingly simple situations very complicated in the matrix setting.

Nevertheless, in Section \ref{sec:deg}, we generalize several parts of Theorem \ref{t-rank1} to matrices. First, we show that if $\Theta$ is a rational inner function of a particular degree, then its associated commutator will have finite rank. The details are as follows.

 \begin{theorem*} \ref{thm:gen1}. If $\Theta$ is a $d\times d$ matrix-valued rational inner function with $\deg \Theta \le (1,n),$ then
 \[
 \Rank [S_{z_1}, S_{z_1}^*] \le dn.
 \]
 \end{theorem*}

We also study the other direction of Theorem \ref{t-rank1}. Here, several of the scalar arguments completely break down in the matrix setting. Still, we are able to conclude that if the commutator has finite rank, then a certain object associated to $\Theta$ is also finite. Specifically, we conclude the following:

 \begin{theorem*} \ref{thm:gen2}. Assume that  $\Theta$ is a $d\times d$ matrix-valued inner function with $\Rank [S_{z_1}, S^*_{z_1}] = n$. Then 
 \[  \dim \mathcal{H}(\Kmax_1) \le n.\]
 \end{theorem*}

Notice that, if $\Theta$ is rational inner, then this result paired with Theorem \ref{thm:rational2} says that $\Rank[S_{z_1},S^*_{z_1} ] =n$ implies that $\deg_2 \det \Theta \le n.$ This is much more in line with Theorem \ref{t-rank1}.
It is worth noting that  the proof of Guo--Wang's result also uses formulas specific to scalar-valued rational inner functions. For this reason, it is not immediately clear how to generalize their proofs to the matrix setting.

For several reasons, the results obtained in Theorems \ref{thm:gen1} and \ref{thm:gen2} are unsatisfactory. These reasons are discussed in detail in Remark \ref{rem:con} and lead us to the following conjecture:

 \begin{conjecture*} \ref{conjecture}. Let $\Theta$ be a $d \times d$ matrix-valued inner function on the bidisk. Then $\Rank [S_{z_1}, S_{z_1}^*] = n$ if and only if $\deg_1 \Theta \le 1$ and $\deg_2 \det \Theta = n.$ \end{conjecture*}

This conjecture is supported by several nontrivial examples detailed in Section \ref{sec:ex}.

  \subsection{Products of One-Variable Functions}\label{ss-products}
 In the scalar setting, if $\theta(z) = \phi(z_1) \psi(z_2)$ is a product of one variable inner functions, then the properties of $S_{z_j}$ and its commutator $[S_{z_j}, S^*_{z_j}]$ are well-understood for $j=1,2.$ Specifically, see \cite[Section 2]{BLPreprint} for results concerning the reducing subspaces, essential normality, and spectrum of these operators. The obtained results rest on the simple decompositions
  \[
  \begin{aligned}
  1- \phi(z_1) \psi(z_2) \overline{\psi(w_2) \phi(w_1)}
  & = \left( 1 - \psi(z_2)\overline{\psi(w_2)} \right) + \psi(z_2) \left( 1 -\phi(z_1) \overline{\phi(w_1)} \right) \overline{\psi(w_2)} \\
  & =  \left( 1 - \phi(z_1)\overline{\phi(w_1)} \right) + \phi(z_1) \left( 1 -\psi(z_2) \overline{\psi(w_2)} \right) \overline{\phi(w_1)}. 
  \end{aligned}
  \]
  Using these,  one can obtain nice formulas for the reproducing kernels of of the shift-invariant subspaces $\Smax_1, \Smin_1$ and $\Smax_2, \Smin_2$ of $\mathcal{K}_{\theta}.$ Most results follow from studying $S_{z_j}$ and $[S_{z_j}, S^*_{z_j}]$ on these well-understood subspaces.
    
  For $\Theta$ a matrix-valued product of one-variable inner functions, this method no longer works. Indeed, non-commutativity implies that such $\Theta$ could be of the form
  \begin{equation} \label{eqn:gentheta} \Theta(z) = \prod_{i=1}^N \Phi_i(z_1) \Psi_i(z_2),\end{equation}
with no apparent simplification.  Even in the simplest case $\Theta(z) = \Phi(z_1) \Psi(z_2)$, finding  reproducing kernels for $\Smax_j$ and $\Smin_j$ is complicated. Indeed, non-commutativity means that, in general,
  \[
  \begin{aligned}
  I- \Phi(z_1) \Psi(z_2) \Psi(w_2)^* \Phi(w_1)^*
  & \ne  \left( I - \Psi(z_2)\Psi(w_2)^* \right) + \Psi(z_2) \left( 1 -\Phi(z_1) \Phi(w_1)^* \right) \Psi(w_2)^*. 
  \end{aligned}
  \]
 Because of this, it is not clear how to obtain reproducing kernel formulas for the spaces $\Smax_1$ and $\Smin_2.$  In contrast, for this particular $\Theta$, the symmetric factorization does hold and so it is possible to write down formulas for the kernels of $\Smax_2$ and $\Smin_1.$ These formulas follow from the characterizations of $\Smax_j$ and $\Smin_j$ in \cite[Proposition 2.1]{bk13}.  However, for the more general $\Theta$ given in \eqref{eqn:gentheta}, it is not clear how to obtain formulas for any of the subspaces. Without the reproducing kernel formulas for $\Smax_j$ and $\Smin_j$, many of the proofs from \cite[Section 2]{BLPreprint} establishing results about the $S_{z_j}$ and $[S_{z_1}, S^*_{z_j}]$ do not generalize. This motivates the question
 
 \begin{question} Is there a method for determining the reproducing kernel formulas for $\Smax_j$ and $\Smin_j$ when $\Theta$ is a $d\times d$ matrix-valued inner function of the form \eqref{eqn:gentheta}?
  \end{question}
  
 The previous question may be asking too much. Indeed, it may be possible to establish certain results, such as the characterization of reducing subspaces, without establishing concrete formulas for the reproducing kernels. This seems especially possible since various characterizations of the spaces $\Smax_j$, $\Smin_j$ were obtained in \cite{bk13} for matrix-valued $\Theta.$ This leads to the general question:
 
  \begin{question} Do any of the results about $S_{z_j}$ and $[S_{z_j}, S^*_{z_j}]$ from \cite[Section 2]{BLPreprint} generalize to case where $\Theta$ is a $d\times d$ matrix-valued inner function of form \eqref{eqn:gentheta}?
  \end{question}
  
  These open questions indicate the complexity of many seemingly-simple problems in the matrix setting.
  
 \section{Relationship Between Degree of $\Theta$ and rank of $[S_{z_j}, S^*_{z_j}]$} \label{sec:deg}
In \cite{BLPreprint}, the authors proved Theorem \ref{t-rank1} by exploiting
connections between the degree of $\Theta$ and the structure of related subspaces $\mathcal{H}(\Kmax_j)$ and $\mathcal{H}(\Kmin_j)$. The needed connections are detailed in \cite[Theorem 3.2]{BLPreprint}. These connections do generalize to matrix-valued inner functions. To state them, recall that if $\Theta$ is a rational inner $d\times d$ matrix-valued function, then we can write 
\[ \Theta(z) = \frac{1}{p(z)} Q(z), \]
where the polynomial $p(z)$ is the least common multiple of the denominators of the entries of $\Theta$ after each entry is put in reduced form and $Q(z)$ satisfies
\[ Q(\tau) Q(\tau)^* = Q(\tau)^* Q(\tau) = |p(\tau)|^2 I \qquad \text{ for a.e.~}\tau \in \mathbb{T}^2.\]

Given this representation, we can state the following result, which generalizes \cite[Theorem 3.2]{BLPreprint} to matrix-valued inner functions. The proof is in \cite{bk13}; the degree bounds appear in \cite[Theorem $1.7$]{bk13} and dimension results appear in \cite[Theorem $1.8$]{bk13}.

\begin{theorem} \label{thm:rational2} Let $\Theta = \frac{Q}{p}$ be a $d\times d$ matrix-valued rational inner function of degree $(m,n)$. Then 
\[ 
\begin{aligned}
\dim \mathcal{H}(\Kmax_1) &= \dim \mathcal{H}(\Kmin_1) = \deg_2 \det \Theta, \\
\dim \mathcal{H}(\Kmax_2) &= \dim \mathcal{H}(\Kmin_2) = \deg_1 \det \Theta.
\end{aligned}
\]
Furthermore, if $f$ is a function in $\mathcal{H}(\Kmax_1) $ or $\mathcal{H}(\Kmin_1)$ then $f = \frac{q}{p}$ where $\deg q \le (m,n-1)$ and  if $g$ is a function in $\mathcal{H}(\Kmax_2) $ or $\mathcal{H}(\Kmin_2)$ then $g = \frac{r}{p}$, where $\deg r \le (m-1,n).$
\end{theorem}

We use this result to obtain the following generalization of one direction of Theorem \ref{t-rank1}: 
 
 \begin{theorem} \label{thm:gen1} If $\Theta$ is a $d\times d$ matrix-valued rational inner function with $\deg \Theta \le (1,n),$ then
 \[
 \Rank [S_{z_1}, S_{z_1}^*] \le dn.
 \]
 \end{theorem}
 
 The proof is similar to that of the corresponding direction in Theorem \ref{t-rank1}. For the convenience of the reader, we include some details.
 
 \begin{proof} Let $\Theta$ be a $d\times d$ matrix-valued rational inner function with $\deg \Theta \le (1,n)$ and let $N =  \deg_2 \det\Theta$ and $M = \deg_1 \det \Theta$. Notice that $N\le dn$ and $M\le d$. Theorem \ref{thm:rational2} with $m \le 1$ informs us that we can find vector-valued functions $f_i,$ $i=1, \dots, N$ with $\deg f_i \le (1,n-1)$ and $g_j$, $j=1, \dots, M$ with $\deg g_j \le (0,n)$ such that 
\[ 
\Kmax_1(z,w) = \sum_{i=1}^N \frac{f_i(z) f_i(w)^*}{p(z)\overline{p(w)}}  \ \text{ and } \  \Kmin_2(z,w) =   \sum_{j=1}^M\frac{ g_j(z) g_j(w)^*}{p(z)\overline{p(w)}}.\]
 Without loss of generality, we assume orthogonality and normality (or trivial norms) of $\left\{ \frac{f_i}{p}\right\}$, and likewise for $\left\{\frac{g_j}{p}\right\}$. Then, since $\mathcal{K}_{\Theta} = \Smax_1 \oplus \Smin_2$, we can write the reproducing kernel of $\mathcal{K}_{\Theta}$ as the sum of the reproducing kernels $K^1_w(z)$ and $K^2_w(z)$ of the spaces $\Smax_1$ and $\Smin_2$ as follows
 \[ \frac{I - \Theta(z) \Theta(w)^*}{(1-z_1\bar{w}_1)(1-z_2\bar{w}_2)} = K_w^1(z) + K_w^2(z) =
   \frac{ \sum_{i=1}^N f_i(z) f_i(w)^*}{p(z)\overline{p(w)}(1-z_1 \bar{w}_1)} +  \frac{ \sum_{j=1}^Mg_j(z) g_j(w)^*}{p(z)\overline{p(w)}(1-z_2 \bar{w}_2)}. \]
Now, fix $e \in \mathbb{C}^d$ and $w\in \mathbb{D}^2$. Using the structures of $K_w^1(z)e$ and $K_w^2(z)e$, we establish formulas for $[S_{z_1}^*, S_{z_1}]K_w^je.$ As the proofs are quite technical and follow the scalar arguments from \cite{BLPreprint} closely, we omit the details. Here are the obtained formulas
\[ 
  \left[S^*_{z_1}, S_{z_1}\right] K^1_w e  =    P_{\Theta} \left(  \sum_{i=1}^{N}\frac{f_i(0,z_2)}{p(0,z_2)} \left( \frac{ f_i(w)}{p(w)} \right)^*e \right)\]
and similarly
  \[  \left[ S^*_{z_1}, S_{z_1} \right] K^2_w e = P_{\Theta}  \left(
 \sum_{i=1}^N \frac{ T_{\bar{z}_1}f_i(z)}{p(0,z_2)} \left( T_{\bar{z}_1} \tfrac{f_i}{p}(w) \right)^* e \right),  \] 
where $P_{\Theta}$ denotes the projection onto $\mathcal{K}_{\Theta}.$  Combining these two formulas shows that
\begin{align*}
 &[S^*_{z_1}, S_{z_1}] \frac{ I- \Theta(z) \Theta(w)^*}{(1-z_1 \bar{w}_1)(1-z_2\bar{w}_2)} \,e =
 \\
 &
 P_{\Theta}  \left(  \frac{1}{p(0,z_2)}\sum_{i=1}^N T_{\bar{z}_1}f_i(z)  \left(T_{\bar{z}_1} \tfrac{f_i}{p}(w) \right)^* e + f_i(0,z_2)\left( \tfrac{ f_i(w)}{p(w)} \right)^* e \right).
 \end{align*}
 Since $\deg f_i \le (1, n-1)$,  then $\deg T_{\bar{z}_1} f_i \le (0, n-1)$ and  $\deg f_i(0, z_2) \le (0, n-1)$. Thus, considering all $w\in \mathbb{D}^2$ and $e \in \mathbb{C}^d$, the set of vector-valued functions of the form 
 \[  \frac{1}{p(0,z_2)}\sum_{i=1}^N T_{\bar{z}_1}f_i(z)  \left(T_{\bar{z}_1} \tfrac{f_i}{p}(w) \right)^* e +f_i(0,z_2)\left( \tfrac{ f_i(w)}{p(w)} \right)^* e\]
 can have at most dimension $nd.$ By the definition of $\mathcal{K}_{\Theta}$,  linear combinations of functions of the form  \[ \frac{ I- \Theta(z) \Theta(w)^*}{(1-z_1 \bar{w}_1)(1-z_2\bar{w}_2)} e \]
 are dense in $\mathcal{K}_{\Theta}$. Thus, we can immediately conclude that 
 \[ \Rank [ S_{z_1}, S_{z_1}^*] \le n d,\]  
as desired. \end{proof}
 
 We can similarly study the other direction of Theorem \ref{t-rank1} in the matrix setting. 
The following result provides a partial generalization.
  
 \begin{theorem} \label{thm:gen2} Assume that  $\Theta$ is a $d\times d$ matrix-valued inner function with $\Rank [S_{z_1}, S^*_{z_1}] = n$. Then 
 \[  \dim \mathcal{H}(\Kmax_1) \le n.\]
 \end{theorem}
 
 \begin{proof} First, observe that if $f \in \mathcal{S}^{max}_1$, then $z_1 f \in \mathcal{K}_{\Theta}$ and so 
 \[ 
 \left(S_{z_1}S_{z_1}^* - S_{z_1}^* S_{z_1} \right) f = P_{\Theta} \left( z_1 T_{\bar{z}_1} f - f  \right) = -P_{\Theta} \left( f(0,z_2) \right).
 \]
 Now, assume that $\Rank [S_{z_1}, S^*_{z_1}] = n$ and by way of contradiction, assume $\dim \mathcal{H}(\Kmax_1) > n.$ Then there is some nontrivial $f \in \mathcal{H}(\Kmax_1)$ such that $f \in \ker 
  [ S_{z_1}, S_{z_1}^*].$ This means
  \[ P_{\Theta} f(0,z_2) = 0\]
  and so, there is some vector-valued $h \in H_d^2(\mathbb{D}^2)$ such that  $f(0,z_2) = \Theta(z) h(z)$. But then, using basic orthogonality relations, 
  \[ \| f (0,z_2) \|^2_{H^2} = \left \langle f, f(0,z_2) \right \rangle_{H^2} = \left \langle f, \Theta h \right \rangle_{H^2} =0.\]
  Thus, $f(0,z_2) \equiv 0$ and $f(z) = z_1 T_{\bar{z}_1}f(z).$ As $f \in \Smax_1$, this implies $z^k_1 T_{\bar{z}_1}f(z) \in \mathcal{K}_{\Theta}$ for all $k \in \mathbb{N}.$ Thus, we can conclude that $T_{\bar{z}_1}f \in \Smax_1$ and so, $f \in z_1 \Smax_1.$  As $\mathcal{H}(\Kmax_1) = \Smax_1 \ominus z_1 \Smax_1$, we conclude that $f \perp f$ and so $f \equiv 0$, a contradiction.
\end{proof}

\begin{remark} \label{rem:con}
\textnormal{
Theorems \ref{thm:gen1} and \ref{thm:gen2} are unsatisfactory for two reasons. First, in the scalar setting, Theorem \ref{t-rank1} shows that if $ [S_{z_1}, S^*_{z_1}]$ is finite rank, then $\theta$ is a rational function. An important part of that result involves the fact that $\Rank [S_{z_1}, S^*_{z_1}] < \infty$ implies $\deg_1 \theta \le 1.$ Unfortunately, the proof of that result relies on scalar arguments that do not generalize to the matrix setting. Nevertheless, we still conjecture that  $\Rank [S_{z_1}, S^*_{z_1}] < \infty$ implies $\deg_1 \Theta \le 1$ and will discuss this further in the next section.}

\textnormal{
Now assume $\Theta$ is a $d\times d$ rational inner function with $\deg_1 \Theta \le 1.$ In the scalar setting, Theorem \ref{t-rank1} shows that if $\deg_1 \theta\le 1$, then $\Rank [S_{z_1}, S^*_{z_1}] =n$ if and only if $\deg_2 \theta =n.$ Let us consider the matrix analogue of this result encoded in Theorems \ref{thm:gen1} and \ref{thm:gen2}. It says
\[
\begin{aligned}
& \text{ If }  \deg_2 \Theta = n, \text{ then } \Rank  [S_{z_1}, S^*_{z_1}]  \le dn. \\
& \text{ If } \Rank  [S_{z_1}, S^*_{z_1}]  =N,  \text{ then } \deg_2 \det \Theta \le N.
\end{aligned}
\]
For $\Theta$ with $\deg_2 \det \Theta = d \cdot \deg_2 \Theta$, then these results combine to give:
\begin{equation} \label{eqn:con1} \Rank  [S_{z_1}, S^*_{z_1}]  =N  \ \text{ if and only if } \ \deg_2 \det \Theta=N.\end{equation} 
We do not currently have this if and only if condition for all $\Theta$ because in general,
\[ \deg_2 \det \Theta \le d \cdot \deg_2 \Theta,\]
with strict inequality possible. However, we conjecture that \eqref{eqn:con1} is actually true for all $\Theta.$} 
 \end{remark} 

The conjectures discussed in Remark \ref{rem:con} combined with our known results to yield:
 
 \begin{conjecture} \label{conjecture} Let $\Theta$ be a $d \times d$ matrix-valued inner function on the bidisk. Then $\Rank [S_{z_1}, S_{z_1}^*] = n$ if and only if $\deg_1 \Theta \le 1$ and $\deg_2 \det \Theta = n.$  \end{conjecture}

\section{Some Examples} \label{sec:ex}

In this section, we consider several examples supporting Conjecture \ref{conjecture}. We demonstrate that this conjecture is true for $d\times d$ diagonal matrix inner functions.  We then investigate several  $2 \times 2$ \emph{non-diagonal} inner functions and show that the conjecture holds for them as well.

\begin{example} 
\textnormal{Consider the $d\times d$ diagonal matrix function
\[ \Theta(z) = \begin{bmatrix} 
\theta_1(z) & &  \\
& \ddots &  \\
& & \theta_d(z) 
 \end{bmatrix}, 
 \]
where each $\theta_i(z)$ is a scalar two-variable inner function. Then, $\mathcal{K}_{\Theta}$ is the direct sum of the $\mathcal{K}_{\theta_i}$ spaces and so,
\[ \Rank [ S_{z_1}, S_{z_1}^*] \text{ on } \mathcal{K}_{\Theta} = \sum_{i=1}^d \left( \Rank [ S_{z_1}, S_{z_1}^*] \text{ on } \mathcal{K}_{\theta_i} \right).\]
Thus, by Theorem \ref{t-rank1}, $\Rank [ S_{z_1}, S_{z_1}^*] =n$ on $\mathcal{K}_{\Theta}$ if and only if each $\theta_i$ is rational inner with 
\[ \deg_1 \theta_i \le 1 \ \text{ and } \sum_{i=1}^d \deg_2 \theta_i = n.\]
Furthermore, using basic facts about rational inner functions, one can show that 
\[ \deg_2 \det \Theta = \sum_{i=1}^d \deg_2 \theta_i. \]
It follows that $\Rank [ S_{z_1}, S_{z_1}^*] =n$ on $\mathcal{K}_{\Theta}$ if and only if $\deg_1 \Theta = \max_i \deg_1 \theta_i \le 1$ and $\deg_2 \det \Theta = n.$ Thus, Conjecture \ref{conjecture} holds for diagonal matrix-valued inner functions.}

\textnormal{
This further implies that Theorem \ref{thm:gen1} is not sharp. Specifically, consider
\[ \Theta(z) =
\left[
\begin{array}{cc}
z_1z_2 & 0 \\
0 & 1 
\end{array} \right]. \]
Theorem \ref{thm:gen1} implies that $\Rank [ S_{z_1}, S_{z_1}^*] \le 2 \cdot 1 = 2.$ However, as $\deg \det_2 \Theta =1$, our earlier arguments show that $\Rank [ S_{z_1}, S_{z_1}^*] = 1.$}
\end{example}

We investigate two examples of inner functions $\Theta$ that are not diagonal.

\begin{example} \textnormal{Consider the matrix-valued function 
\[ \Theta(z) = \frac{1}{2} \begin{bmatrix} 
z_1 +z_2 & z_1-z_2 \\
z_1-z_2 & z_1+z_2 
\end{bmatrix}.\]
A simple computation shows that $\Theta$ is unitary-valued on $\mathbb{T}^2$ and hence, is inner. Observe that we can decompose the reproducing kernel of $\mathcal{K}_{\Theta}$ as follows
\[ 
\frac{I - \Theta(z) \Theta(w)^*}{(1-z_1\bar{w}_1)(1-z_2\bar{w}_2)} = \frac{\tfrac{1}{2}} {1-z_1\bar{w}_1} \begin{bmatrix} -1 \\ 1 \end{bmatrix} \begin{bmatrix} -1 & 1 \end{bmatrix} + \frac{ \tfrac{1}{2}}{1-z_2 \bar{w}_2}  \begin{bmatrix} 1 \\ 1 \end{bmatrix} \begin{bmatrix} 1 & 1 \end{bmatrix}. 
\]
This reproducing kernel decomposition induces the following orthogonal decomposition
\[ \mathcal{K}_{\Theta} =  \begin{bmatrix} -1 \\ 1 \end{bmatrix} H^2_1(\mathbb{D}) \oplus
 \begin{bmatrix} 1 \\ 1 \end{bmatrix} H^2_2(\mathbb{D}) = \mathcal{H}_1 \oplus \mathcal{H}_2,\] 
 where $H^2_j(\mathbb{D})$ denotes the one variable Hardy space with independent variable $z_j.$
We turn our attention to $ [ S_{z_1}, S_{z_1}^*].$ We first show that this operator is identically zero on $\mathcal{H}_2.$ Fix an arbitrary $f \in \mathcal{H}_2$, so $f(z) = \begin{bmatrix} 1 \\ 1 \end{bmatrix} g(z_2),$
where $g \in H^2(\mathbb{D})$. Observe that, as $g$ is a function in $z_2$,
\[ S_{z_1} S_{z_1}^*   \begin{bmatrix} 1 \\ 1 \end{bmatrix} g(z_2) = 0.\]
Now we focus on $S_{z_1} f = P_{\Theta} z_1 f.$
Fix an arbitrary element 
\[ H(z) =  \begin{bmatrix} 1 \\ 1 \end{bmatrix} h(z_2) +  \begin{bmatrix} -1 \\ 1 \end{bmatrix} \tilde{h}(z_1) \in \mathcal{K}_{\Theta}, \] 
where $h, \tilde{h} \in H^2(\mathbb{D})$.  A simple computation shows that 
\[
\begin{aligned}
\left \langle z_1 f, H\right \rangle_{H^2} = 
\left \langle \begin{bmatrix} z_1 \\ z_1 \end{bmatrix} g(z_2), \begin{bmatrix} 1 \\ 1 \end{bmatrix} h(z_2) +  \begin{bmatrix} -1 \\ 1 \end{bmatrix} \tilde{h}(z_1) \right \rangle_{H^2} = \left \langle \begin{bmatrix} g(0) \\ g(0) \end{bmatrix}, \begin{bmatrix} -\tilde{h}'(0) \\ \tilde{h}'(0) \end{bmatrix} \right \rangle_{\mathbb{C}^2} = 0. 
\end{aligned}
\]
Thus, we can conclude that $S_{z_1} f  =0$ and hence $S_{z_1}^*S_{z_1} f \equiv0.$  As $f \in \mathcal{H}_2$ was arbitrary, this implies $[ S_{z_1}, S_{z_1}^*]|_{\mathcal{H}_2} \equiv 0.$}

\textnormal{
Let us compute $ [ S_{z_1}, S_{z_1}^*]$ on $\mathcal{H}_1.$ Fix an arbitrary $f \in \mathcal{H}_1$, so $ f = \begin{bmatrix} -1 \\ 1 \end{bmatrix} g(z_1),$
where $g \in H^2(\mathbb{D})$. It is easy to calculate
\[  [ S_{z_1}, S_{z_1}^*] f = \left( S_{z_1} S_{z_1}^*-S_{z_1}^* S_{z_1} \right)  \begin{bmatrix} -1 \\ 1 \end{bmatrix} g(z_1) =   \begin{bmatrix} -1 \\ 1 \end{bmatrix} \left( z_1 T_{\bar{z}_1} g - g \right) =     \begin{bmatrix} g(0) \\ -g(0) \end{bmatrix}. 
\]
From this, we can conclude that the image of $ [ S_{z_1}, S_{z_1}^*]$ on $\mathcal{H}_1$ is $\begin{bmatrix} -1 \\ 1 \end{bmatrix} \mathbb{C}$ and so, \[\Rank [ S_{z_1}, S_{z_1}^*]|_{\mathcal{H}_1}=1.\] Combining this with our result for $ [ S_{z_1}, S_{z_1}^*]|_{\mathcal{H}_2}$ implies that 
\[ \Rank [ S_{z_1}, S_{z_1}^*] =1.\]
We observe that $\deg_1 \Theta = 1$ and as $\det \Theta = 2 z_1 z_2$, we have $\deg_2 \det \Theta=1.$ Thus, Conjecture \ref{conjecture} says $ \Rank [ S_{z_1}, S_{z_1}^*] =1$, which agrees with our computed result.}

\end{example}
 
In our last example we consider what happens when  $\deg_1 \Theta > 1.$ In the scalar setting, this always causes  $ \Rank [ S_{z_1}, S_{z_1}^*] =\infty.$ Conjecture  \ref{conjecture} claims that the same holds true for matrix-valued inner functions. To test this conjecture, focus on the following example:
 
 \begin{example} \textnormal{Consider the matrix-valued function 
\[
 \Theta(z) = \frac{1}{2} \begin{bmatrix} 
z_1(z_1 +z_2) & z_1(z_1-z_2) \\
z_1-z_2 & z_1+z_2 
\end{bmatrix}.
\]
A simple computation shows that $\Theta$ is unitary-valued on $\mathbb{T}^2$ and hence, is inner. 
Observe that we can decompose the reproducing kernel of $\mathcal{K}_{\Theta}$ as follows
\begin{align*}
 &\frac{I - \Theta(z) \Theta(w)^*}{(1-z_1\bar{w}_1)(1-z_2\bar{w}_2)}= \\
 & \frac{\tfrac{1}{2}} {1-z_1\bar{w}_1} \begin{bmatrix} -z_1 \\ 1 \end{bmatrix} \begin{bmatrix} -\bar{w}_1 & 1 \end{bmatrix} + \frac{ \tfrac{1}{2}}{1-z_2 \bar{w}_2}  \begin{bmatrix} z_1 \\ 1 \end{bmatrix} \begin{bmatrix} \bar{w}_1 & 1 \end{bmatrix} + \frac{1}{1-z_2 \bar{w}_2} \begin{bmatrix} 1 \\0 \end{bmatrix} \begin{bmatrix} 1 & 0 \end{bmatrix}. 
 \end{align*}
This reproducing kernel decomposition induces the following orthogonal decomposition
\[ \mathcal{K}_{\Theta} =  \begin{bmatrix} -z_1 \\ 1 \end{bmatrix} H^2_1(\mathbb{D}) \oplus
 \begin{bmatrix} z_1 \\ 1 \end{bmatrix} H^2_2(\mathbb{D}) \oplus \begin{bmatrix} 1 \\0 \end{bmatrix} H^2_2(\mathbb{D}) = \mathcal{H}_1 \oplus \mathcal{H}_2 \oplus \mathcal{H}_3,\] 
 where $H^2_j(\mathbb{D})$ denotes the one variable Hardy space with independent variable $z_j.$
Consider $ [ S_{z_1}, S_{z_1}^*].$ As $\deg_1 \Theta = 2$, Conjecture \ref{conjecture} indicates that the rank of this operator should be infinite. To see why this is true, we consider $ [ S_{z_1}, S_{z_1}^*]|_{\mathcal{H}_3}.$ Fix an arbitrary $f \in \mathcal{H}_3$. Then $f(z) = \begin{bmatrix} 1 \\ 0 \end{bmatrix} g(z_2)$ for some $g\in H^2(\mathbb{D})$. It is immediate that
\[ S_{z_1} S^*_{z_1}   \begin{bmatrix} 1 \\ 0 \end{bmatrix} g(z_2) =0.\]
Now observe that 
\[z_1 f(z) =  \begin{bmatrix} z_1 \\ 0 \end{bmatrix} g(z_2) = \begin{bmatrix} z_1 \\ 1 \end{bmatrix} \frac{g(z_2)}{2} + \begin{bmatrix} z_1 \\ -1 \end{bmatrix} \frac{g(0)}{2} + \begin{bmatrix} z_1 \\ -1 \end{bmatrix} \frac{z_2 T_{\bar{z}_2} g(z_2)}{2}.\]
 The first two terms come from $\mathcal{H}_2$ and $\mathcal{H}_1$ respectively. Simple computations show that the 
 \[ \begin{bmatrix} z_1 \\ -1 \end{bmatrix} \frac{z_2 T_{\bar{z}_2} g(z_2)}{2} \perp \mathcal{H}_1 \oplus \mathcal{H}_2 \oplus \mathcal{H}_3 = \mathcal{K}_{\Theta}. \]
 Thus, we can compute:
 \[ S_{z_1}f = P_{\Theta} \begin{bmatrix} z_1 \\ 0 \end{bmatrix} g(z_2) = \begin{bmatrix} z_1 \\ 1 \end{bmatrix} \frac{g(z_2)}{2} + \begin{bmatrix} z_1 \\ -1 \end{bmatrix} \frac{g(0)}{2}. \] 
 Finally, we can conclude that
 \[ [ S_{z_1}, S_{z_1}^*] f = - S_{z_1}^*\left( \begin{bmatrix} z_1 \\ 1 \end{bmatrix} \frac{g(z_2)}{2} + \begin{bmatrix} z_1 \\ -1 \end{bmatrix} \frac{g(0)}{2}\right) = - \begin{bmatrix} 1 \\ 0 \end{bmatrix} \left( \frac{g(z_2) + g(0)}{2}\right).\]
 From this, it is clear that $\Rank [ S_{z_1}, S_{z_1}^*] = \infty.$}
 \end{example}

\bibliographystyle{amsalpha}

\end{document}